\documentclass[11pt]{amsart}

\usepackage{amssymb,upgreek}

\usepackage[normalem]{ulem}
\usepackage{subcaption}
\usepackage{url}
\usepackage{bm}
\usepackage[left=1in,top=1in,right=1in,bottom=1in,head=.2in]{geometry}

\usepackage[textsize=scriptsize]{todonotes}
\usepackage{color}
\usepackage{mathtools} % for overbraces

\usepackage[colorlinks=true,allcolors=blue]{hyperref}%

\usepackage{fancyhdr}
\usepackage{mathrsfs}
\usepackage[cmtip,all]{xy}

\usepackage{enumerate}

\swapnumbers
\newtheorem{theorem}{Theorem}[section] % numbered theorems, lemmas, etc
\newtheorem{lemma}[theorem]{Lemma}

\newtheorem{corollary}[theorem]{Corollary}
\newtheorem*{theorem*}{Theorem}

\newtheorem*{fcthm*}{Finite Cork Theorem}
\newtheorem*{ccthm*}{Cork Consolidation Theorem}
\newtheorem*{thm1*}{Theorem 1}
\newtheorem*{thm2*}{Theorem 2}
\newtheorem*{lbthm*}{Generalized 4D Lightbulb Theorem}
\newtheorem*{2lbthm*}{Generalized 4D Lightbulb Theorem (restated)}
\newtheorem*{icthms*}{Infinite Cork Theorems}
\newtheorem*{aclemma*}{\ac-Lemma}
\newtheorem*{mclemma*}{Multicork Lemma}
\newtheorem*{multicorktheorem*}{Multicork Theorem}
\newtheorem*{lemma*}{Lemma}
\newtheorem*{corollary*}{Corollary}

\newcommand{\thistheoremname}{}
\newtheorem{genericthm}[theorem]{\thistheoremname}

\theoremstyle{definition}
\newtheorem{definition}[theorem]{Definition}
\newtheorem{remark}[theorem]{Remark}

\newtheorem*{remark*}{Remark}
\newtheorem*{definition*}{Definition}
\newtheorem*{remarks*}{Remarks}
\newtheorem*{addenda*}{Addenda}

\newcommand{\fig}[3]{\begin{figure}\includegraphics[height=#1pt]{#2}#3\end{figure}}

\newcommand{\bit}[1]{\textbf{\textit{#1}}} % {\textit{#1}} %

\newcommand{\id}{\textup{id}}
\newcommand{\pt}{\textup{pt}}

\newcommand{\sto}{\!\!\xymatrix@C=1em{{}\ar@{~>}[r]&{}}\!\!}

\newcommand{\interior}{\textup{int}}
\newcommand{\ac}{\textup{AC}}

\newcommand{\cs}{\mathop\#}

\newcommand{\items}{\begin{itemize}[leftmargin=25pt,rightmargin=5pt]
  \setlength\itemsep{2pt}}
\newcommand{\stopitems}{\end{itemize}}

\author{Patrick Naylor}
\address{Department of Pure Mathematics, University of Waterloo, Waterloo, Ontario, N2L 3G1, Canada}
\email{patrick.naylor@uwaterloo.ca} 
\urladdr{https://patricknaylor.org}

\author{Hannah R. Schwartz}
\address{Princeton University,
Princeton, NJ 08544}
\email{hs25@princeton.edu} 

\thanks{
    PN and HS were supported by the
    Max Planck Institute for Mathematics in Bonn. PN is supported by an NSERC CGS-D scholarship. 
}

\begin{document}

\title{Gluck twisting roll spun knots}

\begin{abstract}
We show that the smooth homotopy $4$-sphere obtained by Gluck twisting the $m$-twist $n$-roll spin of any unknotting number one knot is diffeomorphic to the standard $4$-sphere, for any $m,n \in \mathbb{Z}$. It follows as a corollary that an infinite collection of twisted doubles of Gompf's infinite order corks are standard.  
\end{abstract}

\maketitle

\vskip-.4in
\vskip-.4in

\parskip 2pt

\setcounter{section}{-1}

\parskip 2pt

\section{Introduction and Motivation}\label{intro}

We begin by recalling the notion of a Gluck twist of a $2$-sphere $S$ smoothly embedded in $S^4$, defined by Gluck in \cite{gluck:2-spheres}. For the unit sphere $S^2 \subset \mathbb{R}^3$, let $r_\theta:S^2 \to S^2$ be the diffeomorphism that rotates $S^2$ by an angle of $\theta$ about the $z$-axis. The \bit{Gluck twist of $S^4$ along $S$} is the smooth homotopy $4$-sphere $$\Sigma_S= (S^4-\interior(N)) \cup_\tau N$$ where $N$ is a regular neighborhood of $S$ diffeomorphic to $S^2 \times D^2$, and $\tau$ is the automorphism of $\partial N \cong S^2 \times S^1$ sending $(x, \theta) \mapsto (r_\theta(x), \theta)$. Although Gluck twists can be defined for spheres with trivial normal bundles in any smooth $4$-manifold, here we will only consider Gluck twists on $S^4$.  
 
By work of Freedman \cite{freedman:simply-connected}, Gluck twists are in fact homeomorphic to $S^4$, and so give a family of potential counterexamples to the smooth $4$-dimensional Poincar\'e conjecture. Gluck \cite{gluck:2-spheres} gave the first non-trivial spheres in $S^4$ whose Gluck twists are standard, i.e., diffeomorphic to the $4$-sphere. His examples include the well-known family of spun knots defined by Artin \cite{artin}. Later, both Gordon \cite{gordon} and Pao \cite{pao} independently showed that Gluck twists of twist spun knots, a generalization of spun-knots due to Zeeman \cite{zeeman}, are also standard. 

A further generalization of spun knots was later given by Litherland \cite{litherland}, who defined the family of deform spun 2-knots. These include twist roll spun knots, which are the focus of this paper. With the exception of twist roll spun torus knots, which Litherland \cite{litherland} showed are isotopic to twist spins of the same knot, Gluck twists of twist roll spun knots are notably absent from the list of smooth homotopy $4$-spheres known to be standard. We remedy this, in part, by providing an infinite family of twist roll spins with standard Gluck twists.

\begin{theorem*}
Gluck twists of $m$-twist $n$-roll spins of unknotting number one knots are standard. 
\end{theorem*}

Our proof gives a new method for proving that Gluck twists are standard, using a recent result from \cite{theguys} that relates the number of $1$-handle stabilizations needed to obtain an unknotted surface to the length of a regular homotopy to the unknot. Recently, Gluck twists on roll spun knots arose as ``twisted doubles" of the infinite order corks constructed by Gompf in \cite{gompf:infinite-order}. As a corollary of our main theorem, we show that many of these doubles are standard, partially answering Question $2.2$ of \cite{gompf:infinite-order}. 

\begin{corollary*}
For all $n,m \in \mathbb{Z}$ and any knot $K \subset S^3$ with unknotting number one, Gompf's twisted double $\mathcal{D}_{m,n}(K)$ is standard.
\end{corollary*}

Our proof extends the result of Akbulut \cite{akbulut1}, that each twisted double $\mathcal{D}_{m,\pm 1}(K)$ is standard when $K$ is the figure-eight knot.

\subsection*{Acknowledgements}
Both authors are extremely grateful to the hospitality provided by the Max Planck Institute for Mathematics in Bonn, where this collaboration began. The first author thanks his advisor Doug Park, while the second author thanks her adviser Paul Melvin for many helpful conversations about Gompf's twisted doubles, as well as Bob Gompf, for constructing the twisted doubles that inspired this paper.

%%%%%%%%%%%%%%%%%%%%%%%%%
\section{Preliminaries} 
%%%%%%%%%%%%%%%%%%%%%%%%%

All manifolds and maps we consider (in particular embeddings and isotopies) will be smooth.

%%%%%%%%%%%%%%%%%%%%%%%%%
\subsection{Regular homotopies and stabilizations} \label{reghom} 
%%%%%%%%%%%%%%%%%%%%%%%%%

%
Throughout the following section, let $S$ be a smoothly embedded 2-sphere in $S^4$.
\begin{definition} \label{guidingarc} 
Suppose that $\alpha$ is an embedded arc with endpoints on $S$ whose interior is embedded in $S^4-S$, as in the top of both diagrams of Figure \ref{fig:guidingarc}. Such an arc is called a \bit{guiding arc} for $S$. 
\end{definition}

We use guiding arcs in two different contexts, both described below and featured in Figure \ref{fig:guidingarc}: to specify a stabilization of a $2$-sphere (as in Figure \ref{fig:stab0}), and to define a regular homotopy (as in Figure \ref{fig:fmwm}). 

%%%%%%%%%% Double figure %%%%%%%%%%
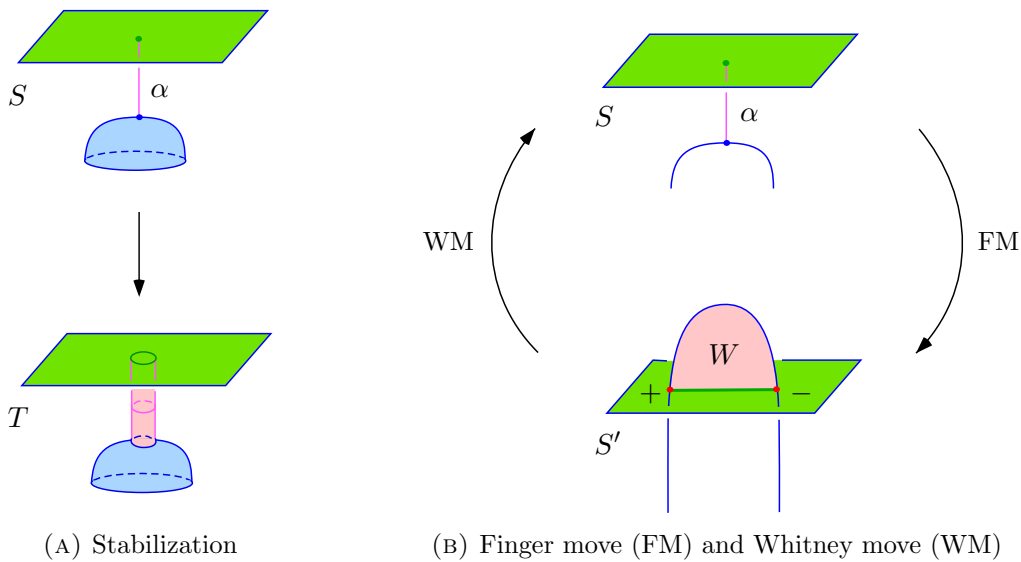
\begin{figure}[ht]
    \begin{subfigure}[t]{.40\textwidth}
      \centering
      % include first image
      \includegraphics[width=.85\linewidth]{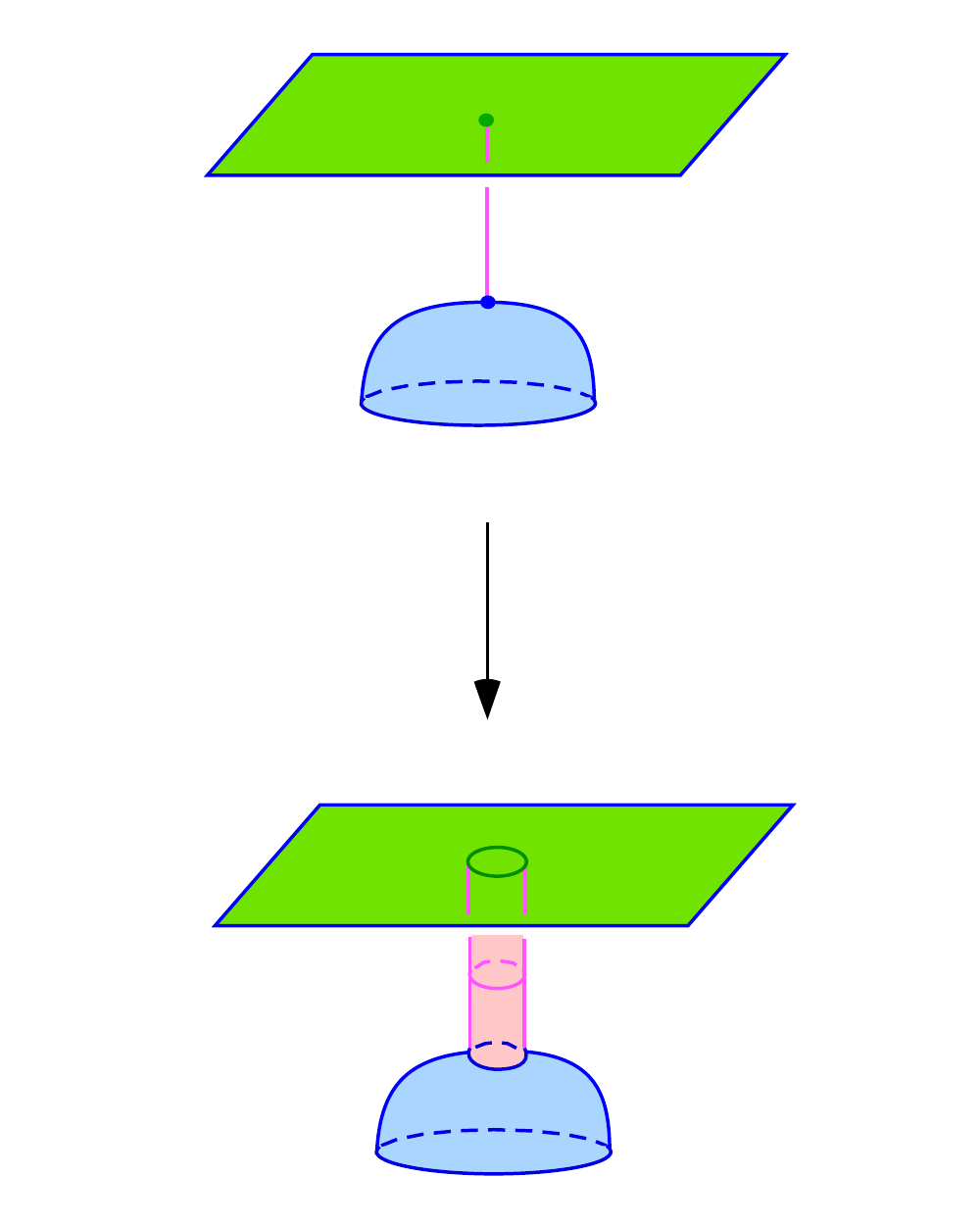}  
     \put(-130,155){$S$}
    \put(-130,33){$T$}
    \put(-76,158){$\alpha$}
      \caption{Stabilization}
      \label{fig:stab0}
    \end{subfigure}
    %%%%%%%%%%%%%
    \begin{subfigure}[t]{.52\textwidth}
      \centering
      % include second image
      \includegraphics[width=1.0\textwidth]{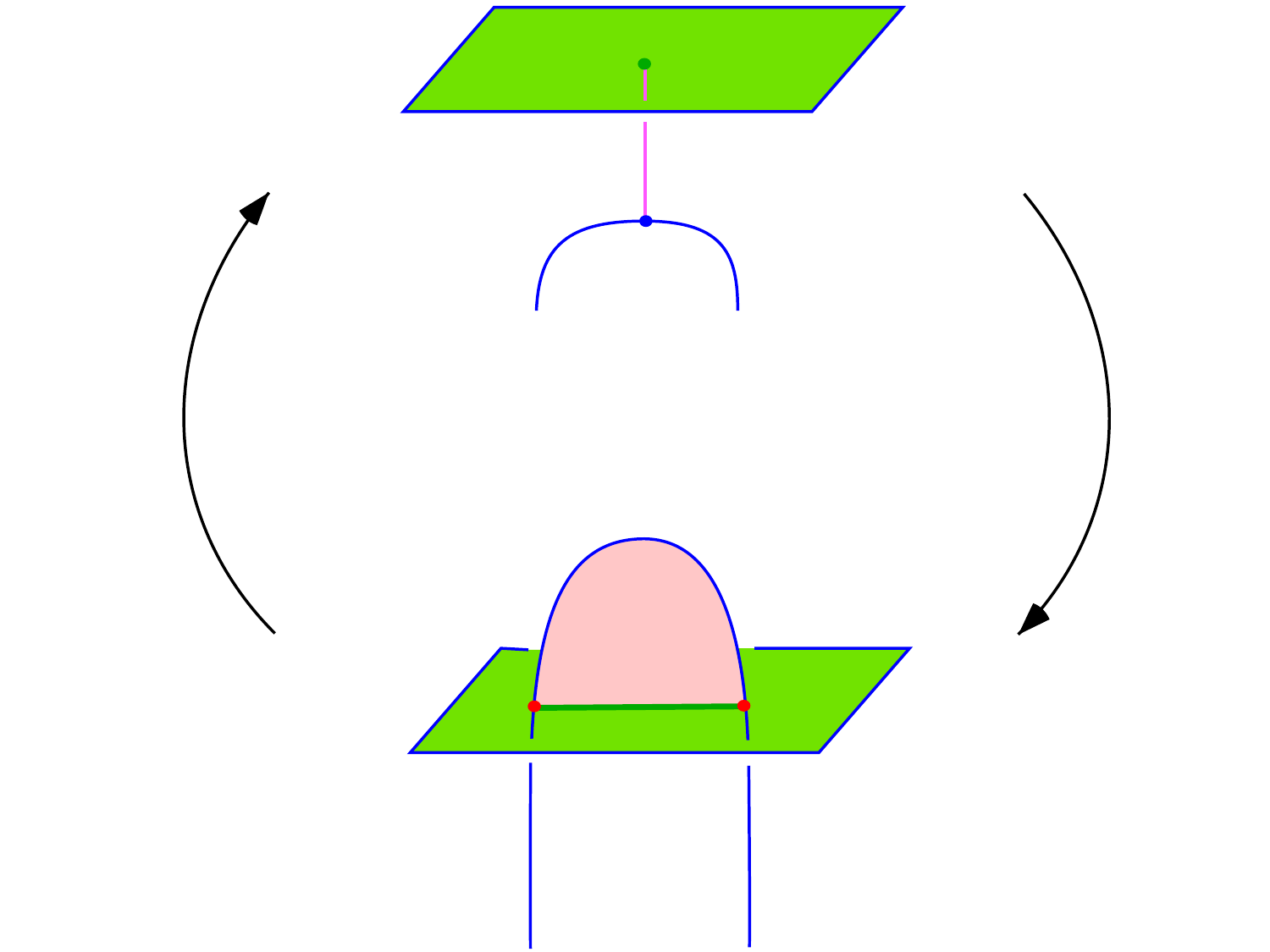}
    \put(-115,149){$\alpha$}
     \put(-170,147){$S$}
    \put(-170,23){$S'$}
    \put(-127,57){$W$}
    \put(-153,43){$+$}
    \put(-96,43){$-$}
    \put(-25,100){\small{FM}}
    \put(-235,100){\small{WM}}
      \caption{Finger move (FM) and Whitney move (WM) \label{fig:fmwm}} 
    \end{subfigure}
\caption{A stabilization (left) and finger move (right) along guiding arc $\alpha$.}
\label{fig:guidingarc}
\end{figure}
%%%%%%%%%%%%%%%%%%%%%%%%%%%

\begin{definition}
Suppose that $\alpha$ is a guiding arc for $S$. An embedded torus $T \subset S^4$ is a \bit{stabilization} of $S$ along $\alpha$ if there is an embedded cobordism in $S^4$ from $S$ to $T$ built from a collar neighborhood $S \times I$ of the sphere $S$ by attaching a single ($3$-dimensional) orientable $1$-handle with core $\alpha$ to $S \times \{1\}$. 
\end{definition}

\begin{definition} 
Suppose $\alpha$ is a guiding arc for $S$. A \bit{finger move} of $S$ along $\alpha$ is a smooth homotopy from $S$ to an immersed sphere $S'$ given by pushing $S$ along the arc $\alpha$ to introduce a pair of oppositely signed double points, as in Figure \ref{fig:fmwm}. The inverse homotopy is called a \bit{Whitney move}, and is supported along an embedded disk (the disk $W$ in Figure \ref{fig:fmwm}) called a \bit{Whitney disk}, whose boundary lies on the immersion $S'$. 
\end{definition} 

By Smale \cite[Theorem A]{smale}, any two smoothly embedded 2-spheres in $S^4$ are \bit{regularly homotopic}, i.e., homotopic through some finite sequence of finger and Whitney moves. The latter homotopy was first defined by its namesake Whitney \cite{whitney} -- later, algebraic and geometric techniques involving the former were developed, largely by Casson \cite{casson1986three}. Frequently, we picture sheets of embedded $2$-spheres in $S^4$ as ``movies'' of arcs in $S^3$, rather than via the schematics of Figure \ref{fig:guidingarc}. Finger and Whitney moves are illustrated from this perspective in Figure \ref{fig:fingerwhitneymove}.

%%%%%%%%%% FIG %%%%%%%%%%
\begin{figure}[ht]
    \centering
    \includegraphics[width=1.0\textwidth]{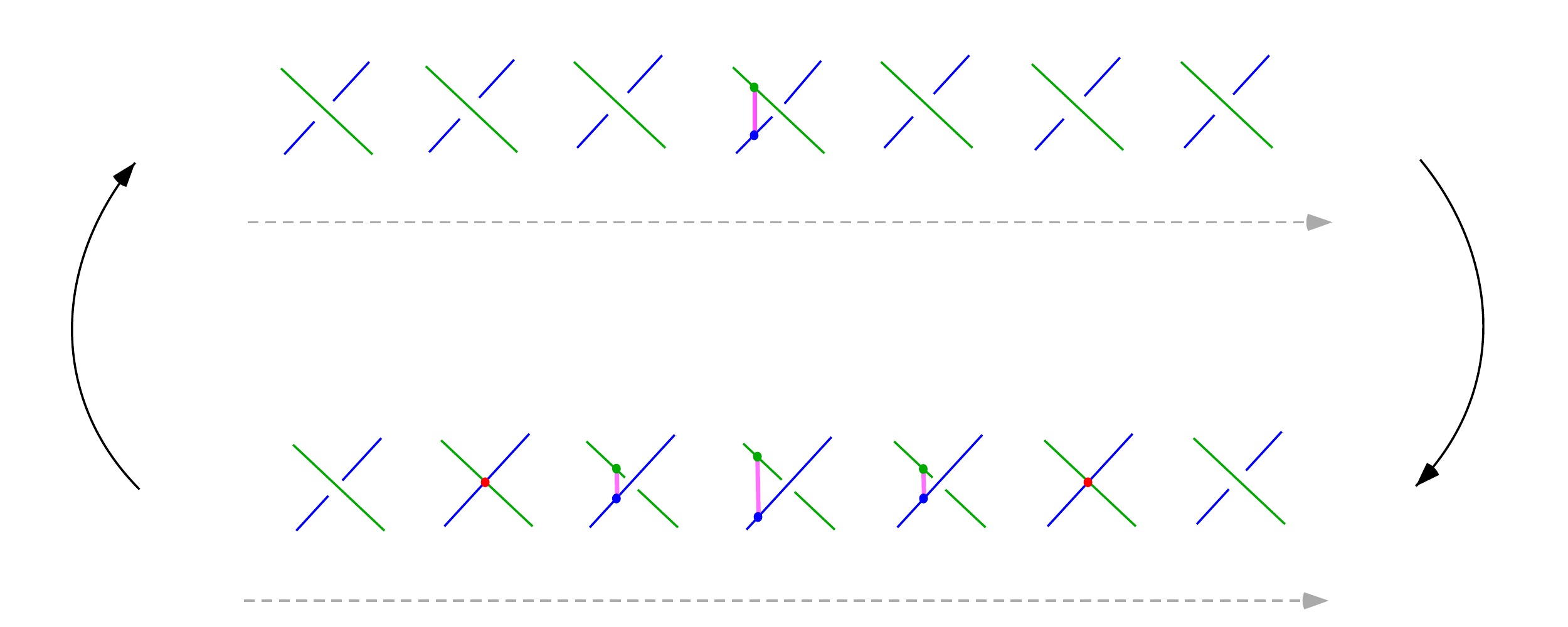}
    \put(-255,153){$\alpha$}
    \put(-258,38){$W$}
    \put(-339,40){$+$}
    \put(-137,41){$-$}
    \put(-20,95){Finger}
    \put(-17,85){move}
    \put(-493,95){Whitney}
    \put(-484,85){move}
    \caption{The spheres before and after a finger move along a guiding arc $\alpha$ (likewise, before and after a Whitney move along $W$). Note that the green and blue arcs in the middle $S^3$ cross-section (from different sheets of the surface) undergo a single crossing change, at the expense of adding cancelling double points.}%
    \label{fig:fingerwhitneymove}
\end{figure}
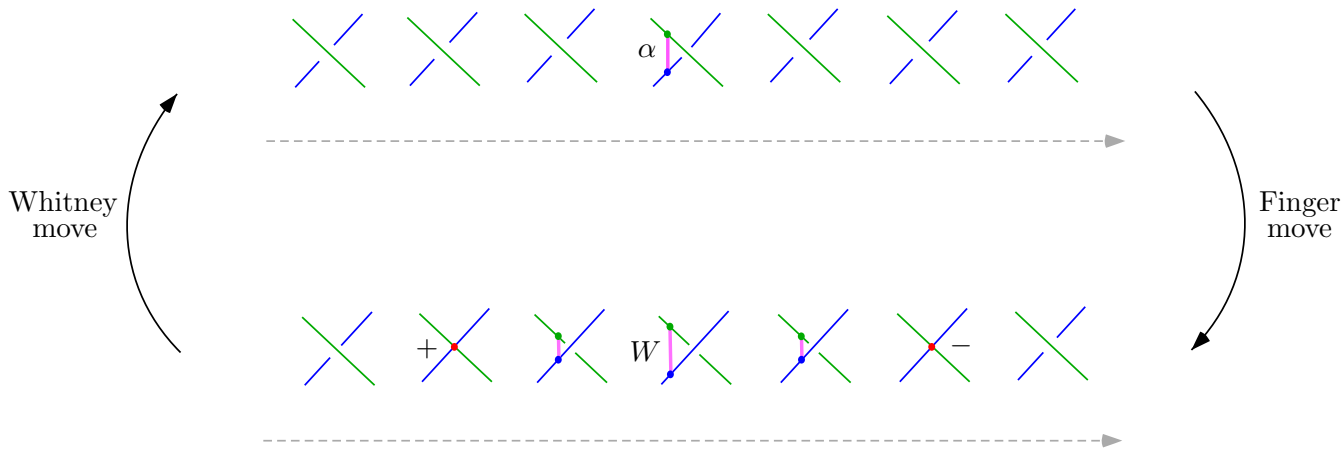
%%%%%%%%%%%%%%%%%%%%%%%%%%%

Simple enough regular homotopies from a $2$-sphere $S \subset S^4$ to the unknotted $2$-sphere give a sequence of stabilizations of $S$ that lead to an unknotted surface (one bounding a solid handlebody). In particular, given a regular homotopy from $K$ to the unknot consisting of one finger and one Whitney move, the unknotted torus can be obtained by a single stabilization of $S$; this was recently shown in \cite{theguys}. 

\begin{theorem}[Joseph-Klug-Ruppik-Schwartz] \label{guys}
Suppose there is a regular homotopy in $S^4$ from a $2$-knot $K$ to the unknot, consisting of one finger and one Whitney move. Then, there is a stabilization of $K$ that is smoothly isotopic to the unknotted torus. 
\end{theorem}

%%%%%%%%%%%%%%%%%%%%%%%%%
\subsection{Gluck twists and logarithmic transformations} \label{sec:gluck}
%%%%%%%%%%%%%%%%%%%%%%%%%

As in the introduction, let $\Sigma_S= (X-\interior(N)) \cup_\tau N$ denote the Gluck twist of a $2$-sphere $S \subset S^4$ with neighborhood $N$ and gluing map $\tau: S^1 \times S^2 \to S^1 \times S^2$. Note that although $\tau$ depends on the initial parametrization of the regular neighborhood $N$ of $S$, this choice is unique up to isotopy by \cite{gluck:2-spheres}. Still, it will occasionally be necessary to specify the two points in $S$ that are fixed by the rotation $r_\theta$ used to define $\tau$. 

\begin{definition}\label{poles}
The ``north and south poles" where $S^2 \subset \mathbb{R}^3$ intersects the $z$-axis are the fixed points of $r_\theta$. The corresponding points in $S$, under the parametrization of $N$, are called the \bit{poles} of the Gluck twist. 
\end{definition}

As noted in the introduction, the Gluck twist $\Sigma_S$ along a $2$-sphere $S \subset S^4$ produces a smooth homotopy 4-sphere. It then follows from work of Freedman \cite{freedman:simply-connected} that $\Sigma_S$ is homeomorphic to $S^4$. Although various families of Gluck twists have been shown to be standard, it remains an open question as to whether \emph{all} Gluck twists are diffeomorphic to $S^4$. However, the Gluck twist along $S$ is known to be standard in many cases\footnote{ A helpful and comprehensive list enumerating families of spheres whose Gluck twists are known to be standard can be found in Sunukjian \cite{sunukjian}.}: for instance, when $S$ is a spun knot, by \cite{gluck:2-spheres}, or when $S$ is a twist spun knot, by \cite{gordon} and \cite{pao}. In section \ref{sec:spinningconstructions}, we give precise definitions along with a further discussion of these spheres. 

Gluck twists on spheres \emph{not} necessarily arising from spinning constructions have also been studied extensively. The Gluck twist $\Sigma_S$ is standard when $S$ is 0-cobordant to the unknot (originally due to \cite{melvin}, and later reproved by \cite{sunukjian}). All ribbon 2-knots are 0-concordant to the unknot; therefore this gives a proof that Gluck twists of ribbon spheres are standard. An exercise in handlebody calculus provides an alternate proof of this fact; see \cite[Ex. 6.2.11(b)]{GS} and \cite{kirby:4-manifolds}. Nash and Stipsicz \cite{nashstipsicz} also use handlebody calculus to give an infinite family of 2-knots which have trivial Gluck twists, obtained by taking the union of two ribbon disks. Hughes, Kim, and Miller \cite{hugheskimmiller} recently showed that certain satellites also have trivial Gluck twists. 

Our main arguments will also involve surgering along embedded tori; this is another well-known ``cut-and-paste" operation on 4-manifolds. To define the construction, suppose that $T$ is an embedded torus in $S^4$, with regular neighbourhood $N$ diffeomorphic to $T^2\times D^2$. 

\begin{definition}\label{logtransform}
A \bit{logarithmic transformation}, or \bit{torus surgery}, on $S^4$ consists of removing the interior of $N$, and regluing it by a diffeomorphism $\phi:T^2\times \partial D^2\to \partial N$. We denote the result of this surgery by $$\Sigma_{T,\phi}=(S^4- \interior(N)) \cup_\phi T^2 \times D^2.$$ 
\end{definition}

As with our definition of the Gluck twist, the definition of a logarithmic transformation generalizes to an embedded torus with a trivial normal bundle in any $4$-manifold. 

Gluing $T^2\times D^2$ to the complement of $\interior (N)$ amounts to attaching a single 2-handle along $\partial N$, followed by two 3-handles, and a 4-handle. Thus it follows by \cite{laupoe} (and the fact that any automorphism of $T^2 \times S^1$ fixing $\{\pt\} \times S^1$ preserves its framing with respect to the product structure) that $\Sigma_{T,\phi}$ is determined up to diffeomorphism by the attaching curve of the $2$-handle. In other words, the only information necessary to specify the surgery is the homology class of $\phi(\{\pt\}\times S^1)$ in $H_1(\partial N;\mathbb{Z}) \cong H_1(T;\mathbb{Z})\oplus \mathbb{Z}$, where the second summand is generated by an oriented meridian to the torus $T$. 

\begin{definition} \label{logtransformmap}
The \bit{multiplicity} of the logarithmic transformation is the integer corresponding to the (oriented) curve $\phi(\{\pt\} \times S^1)$ in $H_1(S^4- N)\cong \mathbb{Z}$.
\end{definition} 

For more details, the reader is encouraged to consult \cite{GS} and \cite{larson}. The next result we recall was originally proven by Montesinos \cite{montesinos} in the early 80's, and later reproved by Kyle Larson \cite{larson} using results of Cerf \cite{cerf}. 

\begin{theorem}[Montesinos, Larson] \label{montesinos} 
Any multiplicity one logarithmic transformation on the unknotted torus in $S^4$ is diffeomorphic to $S^4$. 
\end{theorem}

In addition to this theorem, we will also need a fact due to Iwase \cite{iwase} that every Gluck twist on a sphere $S \subset S^4$ is diffeomorphic to some logarithmic transformation. Implicit in his proof is the fact that these can be taken to be multiplicity one logarithmic transformations of any stabilization of $S$. We present (\`a la Kyle Larson) an alternate geometric proof of this fact that also relies on results of Cerf \cite{cerf}. 

\begin{theorem}[Iwase] \label{glucktostab} 
The Gluck twist $\Sigma_S$ along a $2$-sphere $S \subset S^4$ is diffeomorphic to a multiplicity one logarithmic transformation on any stabilization $T$ of $S$. 
\end{theorem}

\proof Let $\alpha$ be a guiding arc for the stabilization $T$, with neighborhood $N_\alpha$  diffeomorphic to $B^3 \times I$. The intersection $S \pitchfork N_\alpha$ is a pair of disks $D_0 \subset B^3 \times \{0\}$ and $D_1 \subset B^3 \times \{1\}$, with boundaries the unknots $U_1 \subset \partial B^3 \times \{1\}$ and $U_2 \subset \partial B^3 \times \{2\}$. The interiors of these disks can be pushed into the interior of $N_\alpha$, and as noted in Definition \ref{poles}, the endpoints of $\alpha$ can be chosen as the poles of the Gluck twist along $S$. Then, restricted to $N_\alpha$, the Gluck twist removes a neighborhood of both disks $D_0$ and $D_1$, and sews each back ``with a twist" by the map $\tau$ from Definition \ref{poles}.

%%%%%%%%%% FIG 1 %%%%%%%%%%
\fig{200}{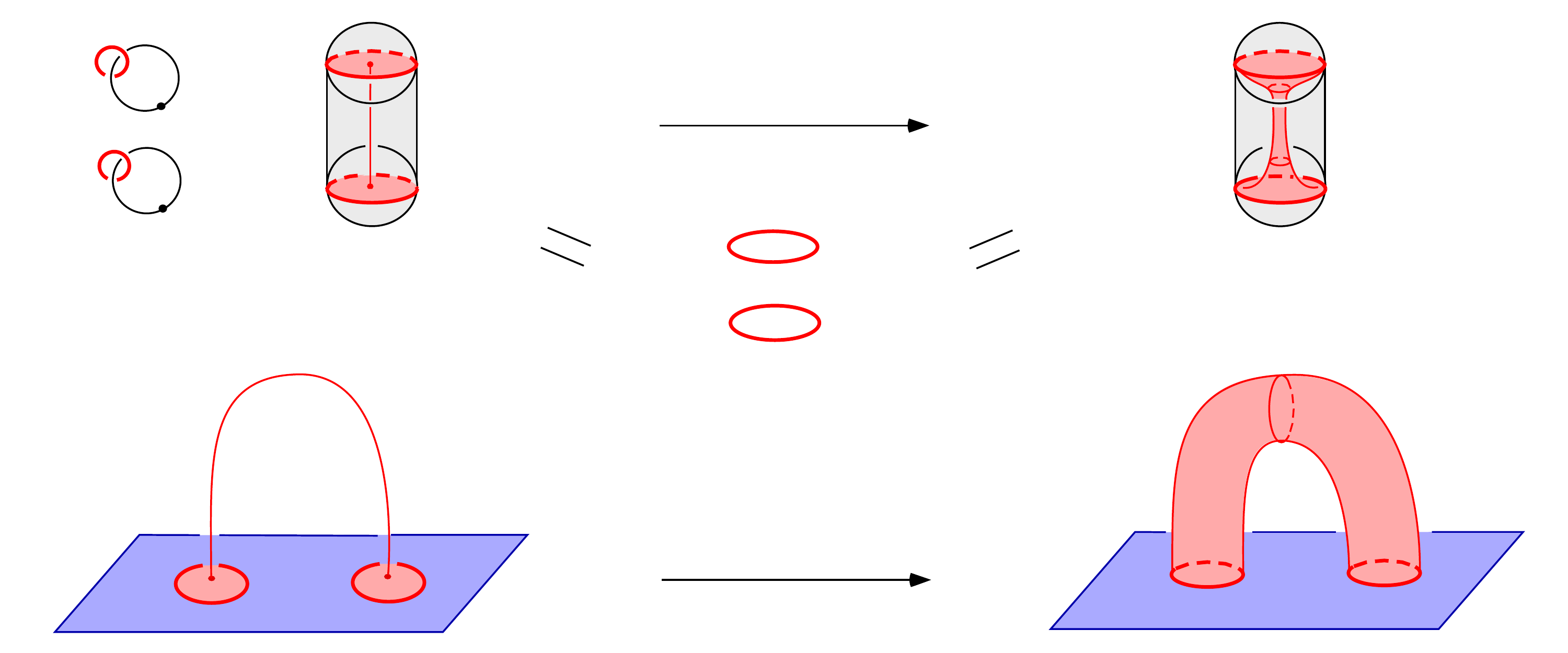}{
\put(-107,115){$B' \subset \Sigma_{T, \phi}$}
\put(-383,115){$B \subset \Sigma_S$}
\put(-450,10){$S$}
\put(-144,12){$T$}
\put(-458,179){\small $1$}
\put(-465,146){\small $-1$}
\put(-410,158){\large $=$}
\put(-402,12){\small $D_0$}
\put(-384,12){\small $D_1$}
\put(-274,178){\small{Replace $4$-ball}}
\put(-264,168){\small{$B$ with $B'$}}
\put(-260,38){\small{Stabilize}}
\put(-248,80){\small{$S^3$}}
\put(-273,98){\small{$-1$}}
\put(-307,131){\small{$\partial$}}
\put(-180,130){\small{$\partial$}}
\put(-266,122){\small{$1$}}
\put(-268,28){\small{the sphere $S$}}
\put(-369,80){$\alpha$}
\caption{As outlined in the proof of Theorem \ref{glucktostab}, removing the $4$-ball $B$ in the Gluck twist $\Sigma_S$ along a sphere $S$ and replacing it with the $4$-ball $B'$ (as illustrated in the top half of the figure) gives the log transform $\Sigma_{T, \phi}$ of $S^4$ along the stabilization $T$ of $S$.}
\label{stab}}
%%%%%%%%%%%%%%%%%%%%%%%%%%%

Note that this submanifold $B \subset \Sigma_S$ can be built from the $4$-ball, as shown on the top left of Figure \ref{stab}, by first adding two $1$-handles; this is diffeomorphic to carving out the disks $D_0$ and $D_1$ from the $4$-ball -- see for example \cite[Chapter I.2]{kirby:4-manifolds}. To sew back in the copy of $D_i \times D^2$ via the map $\tau$, note that the normal disk $\{\ast\}  \times D^2$ to the pole $\{\ast\} \in D_i$ has boundary fixed by $\tau$, and so is sent to a curve passing geometrically once over the corresponding $1$-handle. Therefore, attaching a $2$-handle whose core is this disk is diffeomorphic to sewing back in $D_i \times D^2$. The framing $\pm 1$ of each $2$-handle can be computed by looking at the image under $\tau$ of the boundary of each disk parallel to the core $\{\ast\} \in D_i$. 

Cancelling both of the $1,2$-handle pairs in this handle decomposition for $B$ gives a diffeomorphism from $B$ to the $4$-ball. As indicated in Figure \ref{stab}, there is an  identification of $\partial B$ with the $3$-sphere obtained by performing $\pm 1$ surgery on the unknots $U_1, U_2 \subset S^3$. This surgered copy of $S^3$ also arises as the boundary of the product $B'= B^3_1(U) \times I$, where $B^3_1(U)$ denotes $+1$ surgery on the unknot $U \subset B^3$. The $4$-ball $B'$ can be thought of as the result of removing a tubular neighborhood of the annulus $A = U \times I$ properly embedded in $B^3 \times I$, and sewing it back to effect a $+1$ Dehn surgery on each level $B^3$. Consider the manifold $$\Sigma'_S = (\Sigma_S- \interior(B)) \cup B'.$$ 

By Cerf's theorem \cite{cerf}, the Gluck twist $\Sigma_S$ is diffeomorphic to $\Sigma_S'$, since they differ only by removing the $4$-ball $B$ and replacing it with $B'$. Moreover, the identification between $\partial B$ and $\partial B'$ preserves the surgery curves in $S^3$. Therefore $\Sigma'_S$ may also be described as a logarithmic transformation of $T$ (i.e., $\Sigma_S' = \Sigma_{T,\phi}$ for some gluing map $\phi$ as in Definition \ref{logtransform}). For, the torus $T \subset S^4$ is equal to the union of the annuli $S-(D_0 \cup D_1) \subset S^4-N_\alpha$ and $U \times I \subset B^3 \times I$ along their boundaries. In this particular case, the gluing map $\phi$ restricts to the twist $\tau$ on the boundary of each normal disk $\{\pt\} \times D^2 \subset T^2 \times D^2$. It then follows from Definition \ref{logtransformmap} that this logarithmic transformation is of multiplicity one. \qed 

%%%%%%%%%%%%%%%%%%%%%%%%
\subsection{Twist roll spun knots}\label{sec:spinningconstructions}
%%%%%%%%%%%%%%%%%%%%%%%%

Let $K$ be an embedded knot in $S^3$, and consider a $3$-ball neighborhood $B_p$ of any point $p \in K$ chosen sufficiently small so that $K \pitchfork B_p$ is a single properly embedded trivial arc in $B_p$. Let $k$ denote the properly embedded arc $K \pitchfork B$, where $B = S^3-B_p$. There is a natural slice disk $D_K \subset B^4$ for $K \# -K \subset S^3$ equal to the product $k \times I \subset B \times I$.  

Let $H_K \subset S^3$ denote a tubular neighborhood of the knot $K$ whose boundary $\mathcal T_K \subset S^3$ is the ``swallow-follow torus'' around $K \# -K$ (shown in Figure \ref{fig:swallowfollow}), and let $N_K= \mathcal T_K \times I$ be a collar of the boundary of $H_K$. Identify the handlebody $H_K$ with $S^1 \times D^2$ so that for each $\theta \in S^1$, the curve $\{\theta\} \times \partial D^2$ is a meridian $\mu$ (as in Figure \ref{fig:swallowfollow}) of the knot $K$, and each curve $S^1 \times \{\pt\}$ is a zero-framed longitude $\ell$ of $K$. Consider the diffeomorphisms $\tau_\mu, \tau_\ell : S^1 \times S^1 \times I \to S^1 \times S^1 \times I$ sending $(\theta,\phi,t)$ to $(\theta, \phi + 2 \pi t, t)$ and $(\theta + 2 \pi t, \phi, t)$, respectively. The maps $\rho_\mu, \rho_\ell: S^3 \to S^3$ restricting to the identity on $S^3 - N_K$ and to $\tau_\mu, \tau_\ell$ on $N_K \cong S^1 \times S^1 \times I$ can be thought of as ``Dehn twists along the swallow-follow torus''. 

%%%%%%%%%% FIG %%%%%%%%%%
\begin{figure}[ht]
    \centering
    \includegraphics[width=0.35\textwidth]{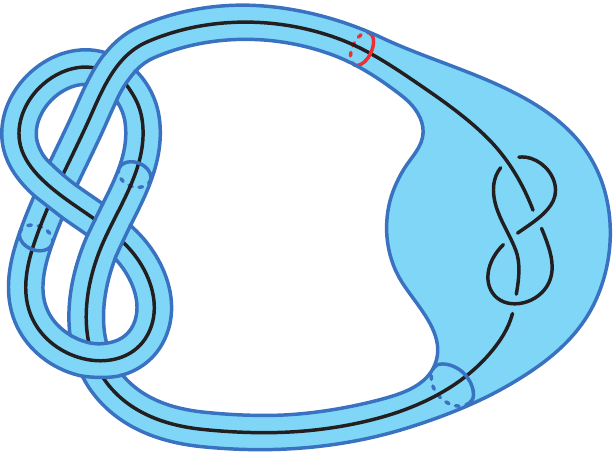}
    \put(-178,57){\small $K$}
    \put(-54,57){\small $-K$}
    \put(-3,25){\small $\mathcal T_K$}
    \put(-76,93){\small $\mu$}
    \caption{The swallow-follow torus $\mathcal T_K \subset S^3$ for $K\#-K$. }%
    \label{fig:swallowfollow}
\end{figure}
%%%%%%%%%%%%%%%%%%%%%%%%%%%

\begin{definition}\label{spin}
The \bit{spin} of $K$, denoted $S(K)$, is the 2-sphere obtained by doubling the pair $(B^4,D_K)$, i.e.,
$$(S^4, S(K)) = (B^4, D_K) \cup_h (S^3 \times I, (K \# -K) \times I) \cup_{h^{-1}} (B^4,D_K),$$
where $h: S^3 \to \partial B^4$ is any diffeomorphism. 
More generally, the \bit{$m$-twist $n$-roll spin} of $K$ is the 2-sphere
$$(S^4, S_{m,n}(K)) = (B^4, D_K) \cup_{h\circ\rho_\mu^{m}\circ\rho_\ell^{n}} (S^3 \times I, (K \# -K) \times I) \cup_{h^{-1}} (B^4,D_K),$$
or equivalently,
$$(S^4, S_{m,n}(K)) = (B^4, D_K) \cup_h (S^3 \times I, C_{m,n}) \cup_{h^{-1}} (B^4,D_K)$$
where $C_{m,n}(K) \subset S^3 \times I$ is the trace of the knot $K \# -K$ under any ambient isotopy from $\id_{S^3}$ to the composition $\rho_\mu^m \circ  \rho_\ell^n$ of the maps defined above. 
\end{definition}

Up to diffeomorphism, the resulting $2$-sphere is independent of the choice of isotopy. However, for visualization purposes we prefer the ambient isotopy induced (via the isotopy extension theorem) by $m$ meridional rotations of the solid torus $H_K$ followed by $n$ rotations in the $S^1$ direction.  
    
As noted in the introduction, the families of spun knots, twist spun knots, and twist roll spun knots are due to \cite{artin}, \cite{zeeman}, and \cite{litherland} respectively. The original constructions ``spin" a knotted arc $k$ through the $3$-ball pages of an open book decomposition of $S^4$ with binding the unknotted $2$-sphere. Our constructions in Definition \ref{spin} via concordances are equivalent, and work better for us in this context. 

\begin{remark} 
By definition, $S_{0,0}(K) = S(K)$. It was shown by \cite{zeeman} that for any knot $K \subset S^3$, the twist spins $S_{\pm 1,0}(K)$ are isotopic to the unknot. If $K \subset S^3$ is \emph{not} a torus knot, then by Teragaito \cite{teragaito-roll}, the roll spin $S_{0,n}(K)$ is non-trivial for any $n$. Furthermore, if $K$ is fibered, then its roll spins are pairwise distinct, and also non-ribbon. In \cite{teragaito-twist}, Teragaito also shows that if $K$ is not a torus knot, $S_{m,n}(K)$ is non-trivial when $|n| \geq 2$. On the other hand, work of Litherland \cite[Corollary 6.4]{litherland} shows that the twist roll spin $S_{m,n}(T_{p,q})$ of a torus knot $T_{p,q}$ is isotopic to $S_{m-npq, 0}(K)$, and hence is non-trivial if and only if $m-npq \not = \pm 1$. 
\end{remark}

%%%%%%%%%% FIG %%%%%%%%%%
\begin{figure}[ht]
    \centering
    \includegraphics[width=0.8\textwidth]{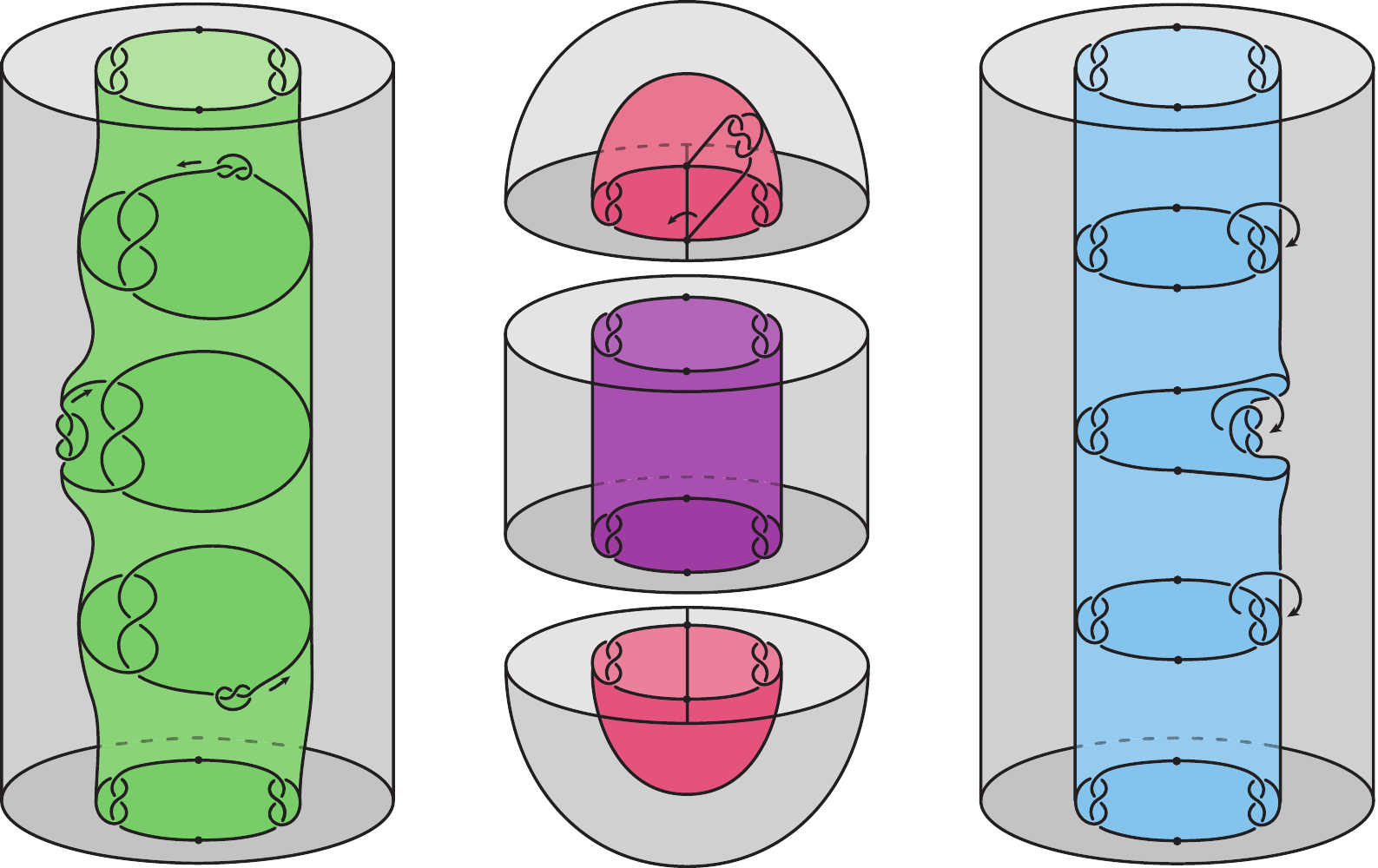}
    \put(-205,115){\small $C_{0,0}(K)$}
    \put(-195,-15){\small $D_K$}
    \put(-205,245){\small $-D_K$}
    \put(-70,-15){\small $C_{1,0}(K)$}
    \put(-340,-15){\small $C_{0,1}(K)$}
    \caption{Schematics for the various spinning constructions: the spin of a knot $K$ viewed as two copies of the slice disk $D_K \subset B^4$ for $K \cs -K \subset S^3$ glued to the product concordance $C_{0,0}(K) \subset S^3\times I$ (center), the twist spin concordance $C_{1,0}(K)$ (right), and the roll spin concordance $C_{0,1}(K)$ (left).  }%
    \label{fig:spinconcordances}
\end{figure}
%%%%%%%%%%%%%%%%%%%%%%%%%%%

Many of the methods used to prove that spun or twist spun knots have standard Gluck twists do not obviously extend to roll spun knots. For instance, Gluck \cite{gluck:2-spheres} observed that a 2-sphere has a standard Gluck twist if it bounds a Seifert 3-manifold in $S^4$ over which the rotation $r_\theta$ of the sphere extends. This shows, in particular, that spun knots have standard Gluck twists. However, it is not clear that twist or roll spun knots bound such $3$-manifolds. 

It also follows that Gluck twists of spun knots are standard because all spun knots are ribbon (see the discussion below Definition \ref{poles}). Yet again, this argument fails to apply to both twist and roll spins, since both twist spun \cite{cochran} and roll spun knots \cite{teragaito-roll} are generally not ribbon. Recently, Sunukjian \cite{sunukjianconc} produced families of twist spins that are not even $0$-concordant to the unknot; more examples were later given by both Dai-Miller \cite{daimiller} and Joseph \cite{jjo}. None of their techniques, however, apply to roll spun knots; hence it remains unknown whether roll spins are 0-concordant to the unknot. This open problem poses an avenue through which to study the Gluck twists of roll spins that is complementary to the techniques we use in this paper.

In \cite{plotnick}, Plotnick  generalized the family of homotopy $4$-spheres obtained by Gluck twisting twist roll spun knots, by removing and regluing a tubular neighborhood of the twist roll spun knot plumbed together with a tubular neighborhood of the unknotted sphere in $S^4$ about which it is spun. However, as he points out below Corollary $6.2$, his argument that many of these manifolds are standard fails for Gluck twists on twist roll spun knots.

%%%%%%%%%%%%%%%%%%%%%%%%%
\section{Main Results} 
%%%%%%%%%%%%%%%%%%%%%%%%%

\begin{theorem}\label{thm:singlefmwm}
Let $S \subset S^4$ be a $2$-sphere with a regular homotopy to the unknot consisting of one finger and one Whitney move. Then, the Gluck twist $\Sigma_S$ is diffeomorphic to the $4$-sphere. 
\end{theorem}

The results from \cite{theguys}, \cite{iwase}, and \cite{montesinos} used in the proof below are stated explicitly in the previous section; see Theorems \ref{guys}, \ref{glucktostab}, and \ref{montesinos} respectively. 

\begin{proof}
There is a stabilization of the sphere $S$ isotopic to the unknotted torus (the one bounding a solid handlebody). This follows from \cite{theguys} using the fact that $S$ is related to the unknot via only a single finger and Whitney move. It then follows from Iwase \cite{iwase} that the Gluck twist $\Sigma_S$ is diffeomorphic to a multiplicity one logarithmic transform of the unknotted torus. Therefore, $\Sigma_S$ is diffeomorphic to the $4$-sphere, by Montesinos \cite{montesinos}. \end{proof}

\begin{remark}\label{rmk:glucktwists}
More generally, suppose that $S, T \subset S^4$ are embedded spheres with a regular homotopy from $S$ to $T$ consisting of one finger and one Whitney move. Then, there is an immersed sphere in $S^4$ with a single pair of double points that can be cancelled by a Whitney move along either a Whitney disk leading to the sphere $S$, or a Whitney disk leading to $T$. Suppose that these Whitney disks share a boundary arc; such a regular homotopy is called ``arc-standard" in \cite{theguys}. In this case, our argument for the theorem above applies to show that the Gluck twists $\Sigma_S$ and $\Sigma_T$ are diffeomorphic. Indeed, the proof of \cite[Theorem A]{theguys} (in particular Figure $9$) shows that there is a torus ``destabilizing'' to the spheres $S$ and $T$ by compressing along different disks with the same boundary. Therefore, the gluing maps of the logarithmic transformations giving $\Sigma_S$ and $\Sigma_T$ are equal. It would be interesting to understand the family of knotted 2-spheres admitting such a sequence of finger/Whitney move pairs leading to the unknot, and also how this family overlaps with the set of spheres $0$-concordant to the unknot (note that the 2-twist spun trefoil admits such a sequence of finger/Whitney move pairs by Lemma \ref{lem:rollspunhomotopy}, but this 2-knot is known not to be 0-concordant to the unknot). 
\end{remark}

\begin{lemma}\label{lem:rollspunhomotopy}
Let $S_{m,n}(K) \subset S^4$ be the $m$-twist $n$-roll spin of a knot $K \subset S^3$ with unknotting number $\mu$. Then, there is a regular homotopy from $S_{m,n}(K)$ to the unknotted $2$-sphere consisting of $\mu$ finger moves followed by $\mu$ Whitney moves. 
\end{lemma}

\begin{proof}
Observe that the spun knot $S_{0,0}(U)$, where $U \subset S^3$ is the unknot, is the unknotted $2$-sphere. We will show that there are regular homotopies consisting of $\mu$ finger moves from both $S_{m,n}(K)$ and $S_{0,0}(U)$ to immersions $S_{m,n}'(K)$ and $S_{0,0}'(U)$ that are ambiently isotopic in $S^4$. Since a Whitney move is the inverse of a finger move, this is sufficient to complete the proof. The level cross sections of the concordances $C_{m,n}(K)$ and $C_{0,0}(U)$ are illustrated in Figure \ref{fig:concordancelevelsets}. Recall that from this point of view, performing a finger move amounts to changing a crossing in some level set, at the expense of adding a pair of algebraically cancelling double points in the nearby level sets (see Figure \ref{fig:fingerwhitneymove}). 

%%%%%%%%%% FIG %%%%%%%%%%
\begin{figure}[ht]
    \centering
    \includegraphics[width=0.7\textwidth]{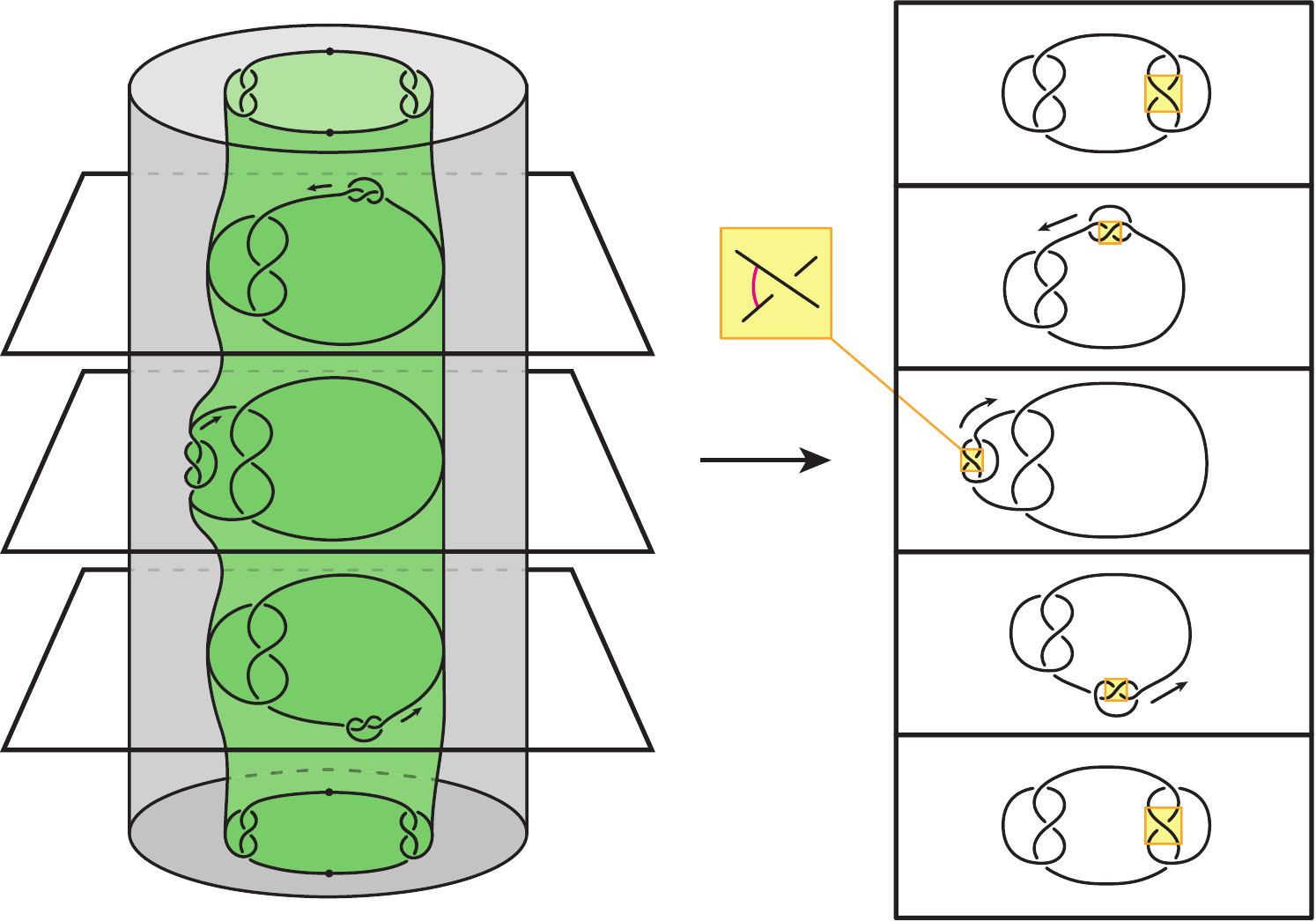}
    \put(-262,-10){\small $C_{0,1}(K)$}
    \put(-138,135){\small $\alpha$}
    \put(-93,203){\small $K$}
    \put(-24,203){\small $-K$}
    \caption{A schematic for the concordance $C_{0,1}(K)$. The trefoil has unknotting number one; the crossing change that unknots it is highlighted, as well as a guiding arc $\alpha$ for this crossing.}%
    \label{fig:concordancelevelsets}
\end{figure}
%%%%%%%%%%%%%%%%%%%%%%%%%%%

Since $K$ has unknotting number $\mu$, there is a projection of $K$ such that $\mu$ crossing changes take $-K$ to the unknot. Let $\alpha_1,\alpha_2, \dots, \alpha_\mu$ be guiding arcs corresponding to these crossing changes (drawn in pink in Figure \ref{fig:concordancelevelsets}) with endpoints on the summand $-K$ in a middle level set of the concordance $C_{m,n}(K)$. Performing finger moves along these guiding arcs has the effect of transforming the $-K$ summand in this level set into the unknot, producing the immersed 2-sphere $S_{m,n}'(K)$ in Figure \ref{fig:movie} (c) with $2\mu$ double points. An ambient isotopy of this immersion moves its positive and negative double points through $S^3 \times I$ to the boundary of the disks $D_K$ and $-D_K$ respectively. After this isotopy, what remains in $S^3 \times I$ is simply the product $(S^3\times I,(K\# U)\times I$). Indeed, the finger moves unknot the summand of $K \cs -K$ that is ``rolled'' and ``twisted'' in the concordance, as in Figures \ref{fig:movie} (b) and (c). 

Now, the $\mu$ crossing changes taking $K$ to the unknot $U$ can be viewed ``backwards'' as crossing changes on some projection of $U$ giving $K$. We can then perform $\mu$ finger moves on $S_{0,0}(U)$ along guiding arcs corresponding to these crossing changes with endpoints on the summand $U$ (rather than $-U$) in a middle level set of the concordance $C_{0,0}(U)$. As above, performing finger moves along these guiding arcs has the effect of transforming the $U$ summand in this level set into the knot $K$, producing the immersed 2-sphere $S_{0,0}'(U)$ in Figure \ref{fig:movie} (a) with $2\mu$ double points. We can isotope the positive and negative double points to the boundary of the disks $-D_U$ and $D_U$ respectively. Again, what is left in $S^3 \times I$ is a product $(S^3\times I, K\cs U)$. Recall from Definition \ref{spin} that the disk $D_K$ is identified with the product $K \times I \subset B^3 \times I$. This product structure allows us to push the double points of the immersion $S_{0,0}'(U)$ through the interior of the disks $\pm D_U$ to lie on the opposite summand of their boundaries, changing the interior of the disks $\pm D_U$ into $\pm D_K$. The result of this isotopy is illustrated in Figure \ref{fig:movie} (a) and (b). This gives an ambient isotopy in $S^4$ from $S_{0,0}'(U)$ to $S_{m,n}'(K)$, as desired. \end{proof}

%%%%%%%%%% FIG %%%%%%%%%%
\begin{figure}[ht]
    \centering
    \includegraphics[width=0.65\textwidth]{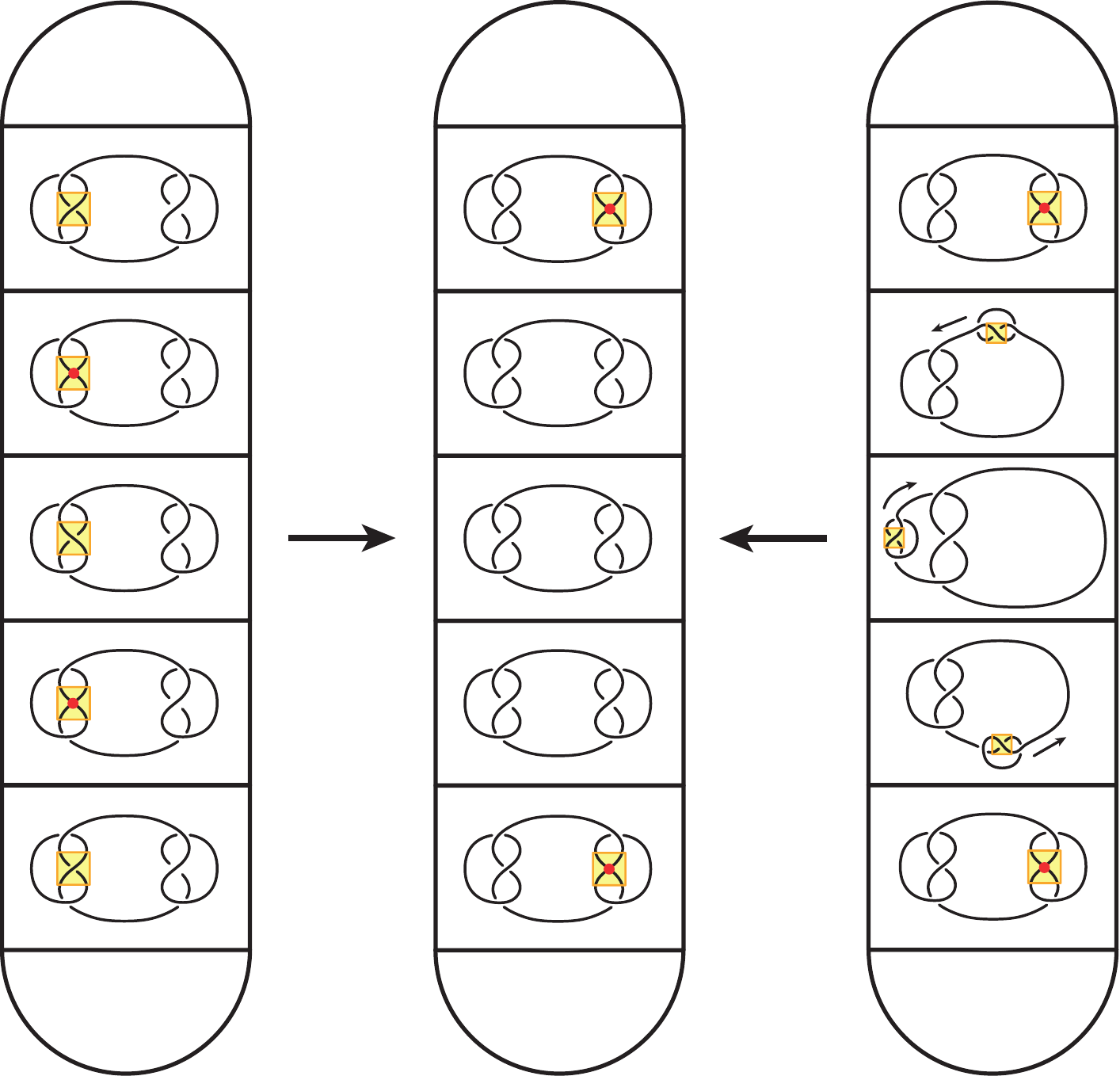}
    \put(-277,-15){(a)}
    \put(-277,17){\small $D_U$}
    \put(-282,270){\small $-D_U$}
    \put(-159,-15){(b)}
    \put(-159,17){\small $D_K$}
    \put(-164,270){\small $-D_K$}
    \put(-41,-15){(c)}
    \put(-41,17){\small $D_K$}
    \put(-46,270){\small $-D_K$}
    \put(-350,143){\small $S_{0,0}'(U)$}
    \put(8,143){\small $S_{m,n}'(K)$}
    \caption{Movies illustrating the three immersions of a 2-sphere in $S^4$ used in the proof of Theorem \ref{lem:rollspunhomotopy}. In (a), $S_{0,0}'(U)$, the result of performing $\mu$ finger moves on the unknotted $2$-sphere $S_{0,0}(U)$. In (c), $S_{m,n}'(K)$: the result of performing $\mu$ finger moves on the twist roll spin $S_{m,n}(K)$. The immersion in (b) is ambiently isotopic in $S^4$ to the immersions in both (a) and (c).}%
    \label{fig:movie}
\end{figure}
%%%%%%%%%%%%%%%%%%%%%%%%%%%

\begin{theorem} \label{main}
For any $n,m \in \mathbb{Z}$ and unknotting number one knot $K \subset S^3$, the Gluck twist of the $m$-twist $n$-roll spin $S_{m,n}(K)$ is diffeomorphic to the $4$-sphere. 
\end{theorem}

\begin{proof}
Since $K$ has unknotting number one, then by Lemma \ref{lem:rollspunhomotopy}, there is a regular homotopy from $S_{m,n}(K)$ to the unknotted 2-sphere consisting of a single finger and Whitney move. Hence by Theorem \ref{thm:singlefmwm}, the Gluck twist of $S_{m,n}(K)$ is diffeomorphic to $S^4$.
\end{proof}

\begin{remark}\label{rmk:higherunknottednumber}
To extend Theorem \ref{main} to a knot $K$ with a higher unknotting number, one might hope that the regular homotopy from Lemma \ref{lem:rollspunhomotopy} can be modified to consist of a sequence of finger/Whitney move pairs (rather than performing all finger moves first, and then all Whitney moves). If such a regular homotopy were arc-standard (as in Remark \ref{rmk:glucktwists} above), it would follow that the Gluck twist of $S_{m,n}(K)$ is standard. However, it is not clear from our construction that this can be achieved, since performing all of the finger moves simultaneously is what allows us to unknot one summand of the connected sum $K \cs -K$. 
\end{remark}

One immediate application of our results stems from recent work of Gompf \cite{gompf:infinite-order}. Given an integer $m \in \mathbb{Z}$ and a knot $K \subset S^3$,  Gompf constructs an infinite family of compact $4$-manifolds $C_m(K)$ by starting with the complement $B^4- N_K$ of a tubular neighborhood $N_K \cong D^2 \times D^2$ of the slice disk $D_K$ from Definition \ref{spin}, and then attaching a $2$-handle with framing $m$ to a meridian of $D_K$ in the boundary of $B^4- N_K$. 
%%%%%%%%%% FIG %%%%%%%%%%
\begin{figure}[ht]
    \centering
    \includegraphics[width=0.75\textwidth]{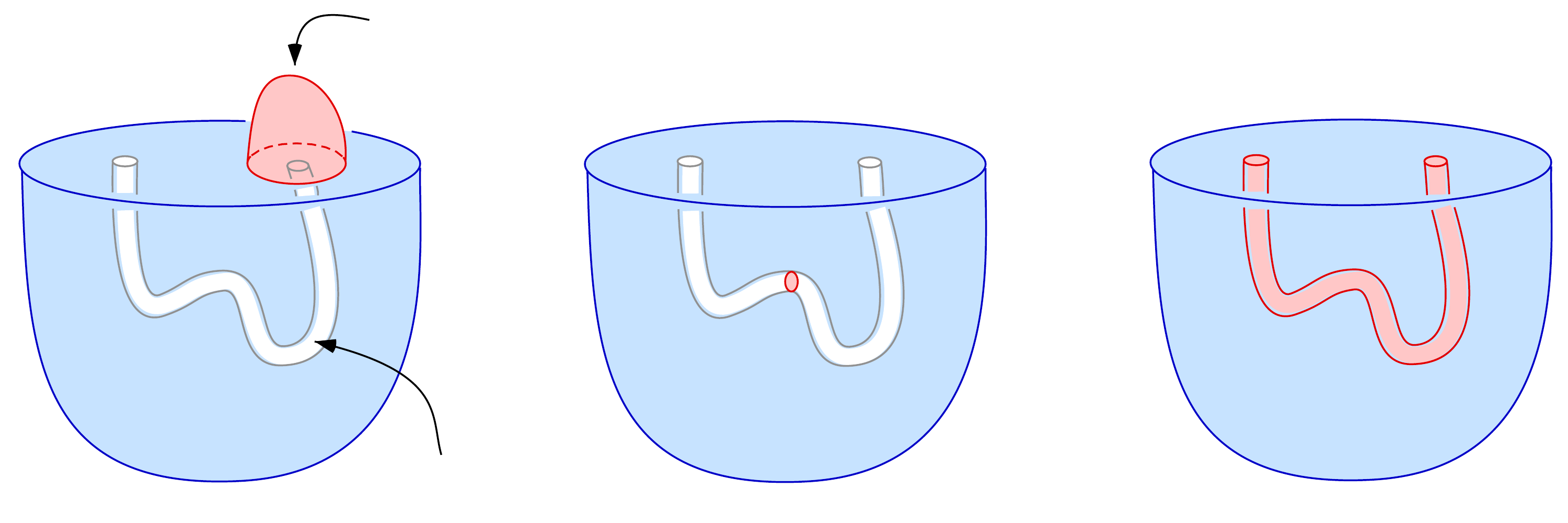}
    \put(-243,50){$=$}
    \put(-116,50){$=$}
    \put(-326,20){\small $B^4-N_K$}
    \put(-258,4){\small $N_K$}
    \put(-265,110){\small $2$-handle}
    \caption{Three equivalent depictions of the manifold $C_m(K)$. In particular, the complement $B^4 - N_K$ of a tubular neighborhood of the slice disk $D_K$ together with: a $2$-handle attached along the meridian of $K \cs -K \subset S^3$ with framing $m$ (left), a $2$-handle attached along the meridian of $D_K$ (center), a copy of $D^2 \times D^2$ glued in by the map $\tau^m$ (right).}%
    \label{fig:gompf}
\end{figure}
%%%%%%%%%%%%%%%%%%%%%%%%%%%
This $2$-handle attachment is equivalent to gluing the lower ``hemisphere'' $D^2 \times S^1 \subset S^2 \times S^1$ to $D^2 \times S^1 \cong  N_K \cap (B^4 - N_K)$ via the restriction of the map $\tau^m$ from Section \ref{sec:gluck}, as illustrated in Figure \ref{fig:gompf}. Each $C_m(K)$ has an associated family of ``twisted doubles''
$$\mathcal D_{m,n}(K) = C_m(K) \cup_{f^n} -C_m(K)$$

\noindent for each integer $n \in \mathbb{N}$, where $f : -\partial C_m(K) \to \partial C_m(K)$ is the ``Dehn twist'' along the longitude of the swallow-follow torus for $K \cs -K$, similar to the map $\rho_\ell$ used in Definition \ref{spin}. 

Since each $C_m(K)$ is contractible, each of its twisted doubles $\mathcal D_{m,n}(K)$ is a homotopy $4$-sphere. In fact, Tange \cite{tange2020} showed that $\mathcal D_{m,n}(K)$ is diffeomorphic to $S^4$ when $m$ is even, and diffeomorphic to the Gluck twist on \emph{some} $2$-sphere $S \subset S^4$ when $m$ is odd. That $S$ is equal to the $n$-roll spin of $K$ was first observed for honest doubles with $n=0$ by Gompf \cite{gompf:infinite-order}, and for certain cases with $n=1$ by Akbulut \cite{akbulut1}. This can be seen in general by constructing $\mathcal D_{m,n}(K)$ in two steps: first, glue two copies of the complement $B^4- N_K$ to each other along the complement of $K \subset S^3$, using the map $f^n$, to obtain the complement of a neighborhood of the $n$-roll spin of $K$. Second, add the remaining two $2$-handles. Since the attaching regions of the $2$-handles are identified in the double, this amounts to gluing in a copy of $S^2 \times D^2$ via the map $\tau^m$. The claim follows since the map $\tau$ has order $2$ by Gluck \cite{gluck:2-spheres}, and so $\tau^m$ is isotopic to the identity when $m$ is even, and isotopic to the map $\tau$ when $m$ is odd.\footnote{We thank Paul Melvin for many discussions about this.}

\begin{corollary}
For all $n,m \in \mathbb{Z}$ and any knot $K \subset S^3$ with unknotting number one, Gompf's twisted double $\mathcal D_{m,n}(K)$ is standard. 
\end{corollary}

\begin{proof}  This follows from the observation above that $\mathcal D_{m,n}(K)$ is diffeomorphic to either $S^4$ or the Gluck twist on the $n$-roll spin of $K$. The latter manifold is standard in this case by Theorem \ref{main} since $K$ has unknotting number $1$. 
\end{proof}

As noted in the introduction, this extends the result of Akbulut \cite{akbulut1} that the homotopy $4$-sphere $\mathcal D_{m,\pm 1}(K)$ is standard for every $m \in \mathbb{Z}$ when $K$ is the figure eight knot.

%%%%%%%%%%%%%%%%%%%%%%%%%%%%%%
\bibliographystyle{alpha}

%%%%%%%%%%%%%%%%%%%%%%%%%%%%%%

\end{document}